\documentclass[11pt,leqno]{amsart}

\usepackage[T1]{fontenc}
\usepackage[utf8]{inputenc}
\usepackage{amsmath,amssymb,amsfonts,mathtools}
\usepackage{microtype}
\usepackage[hidelinks]{hyperref}
\usepackage[left=2.cm,right=2.cm,top=2.cm,bottom=2.cm]{geometry}

\newtheorem{theorem}{Theorem}[section]
\newtheorem{lemma}[theorem]{Lemma}
\newtheorem{proposition}[theorem]{Proposition}
\newtheorem{corollary}[theorem]{Corollary}
\newtheorem{assumption}[theorem]{Assumption}
\theoremstyle{definition}
\newtheorem{definition}[theorem]{Definition}
\newtheorem{example}{Example}
\theoremstyle{remark}
\newtheorem{remark}{Remark}
\numberwithin{equation}{section}

\newcommand{\EP}{\operatorname{EP}}
\newcommand{\VI}{\operatorname{VI}}
\newcommand{\Fix}{\operatorname{Fix}}
\newcommand{\dom}{\operatorname{dom}}
\newcommand{\intdom}{\operatorname{int}\operatorname{dom}}
\newcommand{\argmin}{\operatorname*{arg\,min}}
\newcommand{\Dg}{D_g}

\newcommand{\Bproj}{P^g}
\newcommand{\Bmix}{\mathcal{M}^{g}_{\boldsymbol{\alpha}}}
\newcommand{\Sol}{\mathcal{S}}

\begin{document}

\title[Countable generalized nonexpansive-type mappings]{Strong convergence for countable generalized Bregman nonexpansive-type mappings with equilibrium and variational inequality constraints}

\author{\NoCaseChange{Markjoe O. Uba}\\
\small \NoCaseChange{School of Mathematical and Statistical Sciences, Northern Illinois University, DeKalb, Illinois, USA}\\
\small \NoCaseChange{\texttt{markjoeuba@gmail.com}}}

\date{}

\begin{abstract}
We develop a Legendre--Bregman outer-approximation method for a common-solution problem in a uniformly smooth and uniformly convex Banach space. The constraint system consists of countable families of fixed-point-type mappings, equilibrium bifunctions, and monotone variational-inequality operators. Strong convergence to the Bregman projection of the initial point onto the common solution set is obtained without a family-level NST compatibility condition. A localized theorem covers negative-entropy generators, and Hilbert-space and Alber-functional cases are also derived.
\end{abstract}

\maketitle

\section{Introduction}

Let $E$ be a real Banach space, let $E^*$ be its dual, and let $C\subset E$ be nonempty, closed, and convex. For a mapping $A:C\to E^*$, the associated variational inequality asks for $x^*\in C$ satisfying
\begin{equation}\label{introVI}
\langle y-x^*,Ax^*\rangle\geq0,\qquad y\in C.
\end{equation}
We write $\VI(C,A)$ for the solution set. This formulation goes back to the classical variational-inequality framework of Stampacchia and includes constrained optimization and complementarity models; see \cite{Stampacchia1964}.

For a bifunction $\Theta:C\times C\to\mathbb R$, the equilibrium problem is
\begin{equation}\label{introEP}
\Theta(x^*,y)\geq0,\qquad y\in C,
\end{equation}
and its solution set is denoted by $\EP(\Theta)$. The equilibrium formulation provides a common language for optimization, variational inequalities, saddle-point problems, and related models; see \cite{BlumOettli1994,CombettesHirstoaga2005}.

Strongly convergent hybrid projection methods for common fixed-point, equilibrium, and variational-inequality constraints are often built from Alber's functional
\[
\phi(x,y)=\|x\|^2-2\langle x,Jy\rangle+\|y\|^2,
\]
where $J$ is the normalized duality mapping; see \cite{Alber1996}. The corresponding generalized $J_*$-nonexpansive framework for mappings was used in \cite{UbaCarpathian2023}. Bregman geometry replaces the quadratic functional by
\[
\Dg(x,y)=g(x)-g(y)-\langle\nabla g(y),x-y\rangle,
\]
where $g$ is a Legendre function. This permits the algorithmic geometry to be chosen independently of the norm and includes entropy-type divergences as well as the quadratic Alber case; see \cite{Bregman1967,BauschkeBorwein1997,ButnariuIusem2000,ReichSabach2010}.

The present work belongs to a sequence of related but distinct hybrid constructions. The earlier published scheme \cite{UbaCarpathian2023} uses Alber geometry and treats a countable fixed-point family together with finite equilibrium and variational-inequality families. A companion viscosity--Halpern scheme treats all three families as countable but retains Alber geometry and a family-level NST hypothesis \cite{UbaViscosity2026}. Another companion hybrid scheme adds countably many maximal monotone inclusions, again in Alber geometry \cite{UbaCountableMonotone2026}. The inertial study \cite{UbaInertial2026} develops perturbation-resilient fixed-point iterations in Bregman geometry but does not contain the equilibrium and variational-inequality resolvent blocks considered here.

The contributions of the paper are as follows:
\begin{itemize}
\item the Alber-geometry generalized $J_*$-nonexpansive inequality is extended to a direct Legendre--Bregman formulation for mappings $T:C\to E^*$;
\item a shrinking Bregman projection method is analyzed for countable fixed-point, equilibrium, and variational-inequality families;
\item strong convergence is proved without an NST-type relation between the members of the fixed-point family;
\item a localized convergence theorem permits negative-entropy generators on invariant bounded regions;
\item Hilbert-space and Alber-functional consequences and two entropy examples are given.
\end{itemize}

Section~\ref{sec:problem} gives the setting and algorithm. Section~\ref{sec:main} states the convergence results. Standard Bregman and resolvent facts, together with the paper-specific auxiliary statements, are collected in Section~\ref{sec:prelim}. Section~\ref{sec:examples} contains the entropy examples, and Section~\ref{sec:proofs} contains the proofs.

\section{Problem formulation}\label{sec:problem}

This section specifies the Banach--Bregman setting, the three countable constraint blocks, and the
outer-approximation algorithm. The auxiliary Bregman identities, projection facts, and resolvent
properties used later are collected in Section~\ref{sec:prelim}.

\subsection{Bregman geometry and generalized \texorpdfstring{$J_*$}{J-star}-nonexpansive mappings}

Let $E$ be a uniformly smooth and uniformly convex real Banach space with dual space $E^{*}$. The
duality pairing between $E$ and $E^{*}$ will be denoted by $\langle\cdot,\cdot\rangle$. Let
$J:E\to E^{*}$ be the normalized duality mapping, characterized by
\[
\langle x,Jx\rangle=\|x\|^{2}=\|Jx\|^{2},\qquad x\in E.
\]
Under the standing geometric assumptions on $E$, the standard properties of the normalized duality mapping imply that $J$ is single-valued and bijective; see \cite{Cioranescu1990,Alber1996}. We write
$J_*:=J^{-1}:E^{*}\to E$. Let $g:E\to(-\infty,+\infty]$ be a proper, lower semicontinuous and convex
function. We write
\[
\dom g:=\{x\in E:g(x)<+\infty\},\qquad U:=\intdom g.
\]
Throughout this paper we shall assume that $C\subset U$ is nonempty, closed and convex.

\begin{definition}[Legendre function; see \cite{BauschkeBorwein1997,ButnariuIusem2000}]
Let $g:E\to(-\infty,+\infty]$ be proper, lower semicontinuous, and convex. The function $g$ is called Legendre if it is essentially smooth and essentially strictly convex. In particular, $g$ is differentiable on $U=\intdom g$, and $\nabla g$ is single-valued on $U$.
\end{definition}

For a Legendre function on a reflexive Banach space, Legendre duality gives the bijection
\[
\nabla g:U\longrightarrow U^{*}:=\intdom g^{*}
\]
with inverse $\nabla g^{*}$; see \cite{BauschkeBorwein1997,ButnariuIusem2000}. Since $\dom g^{*}$ is convex, $U^{*}$ is convex. Consequently, every
convex combination of points in $\nabla g(U)$ belongs to $U^{*}$ and may be mapped back to $U$ by
$\nabla g^{*}$.

Following Bregman \cite{Bregman1967}, if $g$ is G\^ateaux differentiable on $U$, the Bregman distance generated by $g$ is defined by
\begin{equation}\label{Dgdef}
\Dg(x,y)=g(x)-g(y)-\langle \nabla g(y),x-y\rangle,
\qquad x\in\dom g,\ y\in U.
\end{equation}
It is well known that $\Dg(x,y)\geq0$ for all $x\in\dom g$ and $y\in U$. However, $\Dg$ is generally
not symmetric and it does not satisfy the triangle inequality.

We shall use the following standing assumptions on $g$.
\begin{itemize}
\item[(G1)] $g$ is a Legendre function which is strongly coercive, meaning that
$g(x)/\|x\|\to+\infty$ as $\|x\|\to\infty$ with $x\in\dom g$, and is bounded on bounded subsets of $U$;
\item[(G2)] $g$ is uniformly Fr\'echet differentiable on bounded subsets of $U$;
\item[(G3)] $g$ is totally convex on bounded subsets of $U$, that is, whenever $\{x_n\}$ and $\{y_n\}$
are bounded sequences in $U$ and $\Dg(x_n,y_n)\to0$, then $\|x_n-y_n\|\to0$;
\item[(G4)] $g^{*}$ is uniformly convex on bounded subsets of $U^{*}=\intdom g^{*}$;
\item[(G5)] for every $x\in U$ and $R>0$, the forward and reverse Bregman sublevel sets
\[
\{z\in U:\Dg(z,x)\leq R\},
\qquad
\{z\in U:\Dg(x,z)\leq R\}
\]
are bounded; moreover, $\nabla g$ maps these sublevel sets into bounded subsets of $E^{*}$, while
$\nabla g^{*}$ maps bounded subsets of $U^{*}$ into bounded subsets of $E$ and is uniformly continuous
on bounded subsets of $U^{*}$.
\end{itemize}
By the standard bounded-set differentiability criterion, boundedness of $g$ together with uniform
Fr\'echet differentiability on bounded subsets entails norm-to-norm uniform continuity of $\nabla g$
on those subsets; see, for example, \cite{ButnariuIusem2000}. Consequently, $(G1)$--$(G2)$ yield
uniform continuity of $\nabla g$ on each bounded subset of $U$. In particular, both $g$ and $\nabla g$ are
continuous on $U$, and for every fixed $p\in U$ the map $y\mapsto \Dg(p,y)$ is continuous on $U$.

\begin{remark}
For applications in
which the iteration is known to remain in a bounded region of $U$, such as the truncated-simplex entropy
examples below, only local regularity and convexity are needed. A precise localized counterpart of
Theorem~\ref{mainthm} is stated and proved in Corollary~\ref{localmainthm}.
\end{remark}

\begin{definition}[Bregman projection; see \cite{Bregman1967,AlberButnariu1997,BauschkeBorwein1997}]
Let $D\subset U$ be nonempty, closed, and convex, and let $x\in U$. A \emph{Bregman projection} of $x$ onto $D$ is a minimizer of $z\mapsto\Dg(z,x)$ over $D$. Whenever the minimizer is unique, it is denoted by
\[
\Bproj_Dx:=\argmin_{z\in D}\Dg(z,x).
\]
\end{definition}

\begin{definition}[Weighted Bregman mixing map]\label{barydef}
Let $x_1,x_2,x_3\in U$ and let
$\boldsymbol{\alpha}=(\alpha_1,\alpha_2,\alpha_3)\in(0,1)^3$ satisfy
$\alpha_1+\alpha_2+\alpha_3=1$. The associated three-point Bregman mixing map is
\[
\Bmix(x_1,x_2,x_3)
:=\nabla g^{*}\bigl(\alpha_1\nabla g(x_1)+\alpha_2\nabla g(x_2)+\alpha_3\nabla g(x_3)\bigr).
\]
\end{definition}

\subsection{Equilibrium and variational inequality components}

\begin{definition}[$J$-fixed point; see \cite{UbaCarpathian2023}]\label{Jfixeddef}
Let $T:C\to E^{*}$ satisfy $J_*T(C)\subset C$. A point $p\in C$ is called a $J$-fixed point of $T$ if
\[
Tp=Jp.
\]
The set of all $J$-fixed points of $T$ is denoted by $F_J(T)$.
\end{definition}

The next definition is the direct Legendre--Bregman extension of the generalized $J_*$-nonexpansive inequality used in the Alber-functional setting of \cite{UbaCarpathian2023}.
\begin{definition}\label{genBregJstar}
A mapping $T:C\to E^{*}$ satisfying $J_*T(C)\subset C$ is called generalized Bregman $J_*$-nonexpansive if $F_J(T)\neq\emptyset$ and
\[
\Dg(p,J_*Tx)\leq \Dg(p,x),
\qquad x\in C,\ p\in F_J(T).
\]
\end{definition}

\begin{remark}\label{JstarBregremark}
If $g(x)=\frac12\|x\|^2$, then $\nabla g=J$ and $\Dg=\frac12\phi$. Hence Definition~\ref{genBregJstar} becomes
\[
\phi(p,J_*Tx)\leq\phi(p,x),
\qquad x\in C,\ p\in F_J(T),
\]
which is precisely the generalized $J_*$-nonexpansive condition in Alber geometry used in \cite{UbaCarpathian2023}. We retain the direct notation $J_*Tx$ throughout so that the mapping $T:C\to E^*$ remains explicit in both the algorithm and the convergence proof.
\end{remark}

\begin{definition}[$J_*$-closed mapping; compare \cite{UbaCarpathian2023}]\label{Jstarcloseddef}
A mapping $T:C\to E^{*}$ satisfying $J_*T(C)\subset C$ is called $J_*$-closed if, whenever $x_n\to x$ in $C$ and $J_*Tx_n\to y$ in $E$, one has
\[
y=J_*Tx.
\]
\end{definition}

Let $\Theta:C\times C\to\mathbb R$ be a bifunction. We use the standard equilibrium assumptions of \cite{BlumOettli1994,CombettesHirstoaga2005,TakahashiZembayashi2008}:
\begin{itemize}
\item[(E1)] $\Theta(x,x)=0$ for all $x\in C$;
\item[(E2)] $\Theta$ is monotone, that is,
\[
\Theta(x,y)+\Theta(y,x)\leq0,
\qquad x,y\in C;
\]
\item[(E3)] for all $x,y,z\in C$,
\[
\limsup_{t\downarrow0}\Theta(tz+(1-t)x,y)\leq \Theta(x,y);
\]
\item[(E4)] for each $x\in C$, $\Theta(x,\cdot)$ is convex and lower semicontinuous on $C$.
\end{itemize}

\begin{remark}
If a bifunction is first specified in gradient coordinates as
$\widehat\Theta:\nabla g(C)\times\nabla g(C)\to\mathbb R$, we use only its primal pullback
\[
\Theta(x,y):=\widehat\Theta(\nabla g(x),\nabla g(y)),\qquad x,y\in C.
\]
Thus the theorem requires $(E1)$--$(E4)$ for $\Theta$ on the convex set $C$.
\end{remark}

Motivated by the standard equilibrium and Bregman-resolvent constructions in \cite{CombettesHirstoaga2005,TakahashiZembayashi2008,ReichSabach2010}, for $r>0$ and $x\in C$, let
\begin{equation}\label{Eresolvent}
\mathcal Z_r^{\Theta}(x):=
\left\{z\in C:
\Theta(z,y)+\frac1r\langle y-z,\nabla g(z)-\nabla g(x)\rangle\geq0
\ \text{for every }y\in C\right\},
\end{equation}
and, for a monotone mapping $A:C\to E^{*}$,
\begin{equation}\label{VIresolvent}
\mathcal W_r^{A}(x):=
\left\{z\in C:
\langle y-z,Az\rangle+\frac1r\langle y-z,\nabla g(z)-\nabla g(x)\rangle\geq0
\ \text{for every }y\in C\right\}.
\end{equation}
The notation deliberately distinguishes the candidate sets in \eqref{Eresolvent}--\eqref{VIresolvent} from the resolvent mappings that arise after uniqueness has been established.

\begin{assumption}\label{resolventprops}
For every individual bifunction $\Theta$ and every individual norm-continuous monotone mapping $A$
selected from the countable families below, every $x\in C$, and every $r>0$, the sets
\[
\mathcal Z_r^{\Theta}(x)\quad\text{and}\quad\mathcal W_r^{A}(x)
\]
are nonempty.
\end{assumption}

\subsection{Common-solution problem and iterative method}

Let $\{\Theta_l\}_{l\geq1}$ be a countable family of bifunctions from $C\times C$ into
$\mathbb R$, each satisfying $(E1)-(E4)$. Let $\{A_k\}_{k\geq1}$ be a countable family of
norm-continuous monotone mappings from $C$ into $E^{*}$, and let $\{T_j\}_{j\geq1}$ be a countable
family of $J_*$-closed generalized Bregman $J_*$-nonexpansive mappings from $C$ into $E^{*}$. Define
\begin{equation}\label{commonset}
\Sol:=\left(\bigcap_{j=1}^{\infty}F_J(T_j)\right)
\cap\left(\bigcap_{l=1}^{\infty}\EP(\Theta_l)\right)
\cap\left(\bigcap_{k=1}^{\infty}\VI(C,A_k)\right).
\end{equation}
The preliminary set-structure result in Section~\ref{sec:prelim} shows that $\Sol$ is closed and convex whenever it is nonempty.% We therefore

Let $\alpha_1,\alpha_2,\alpha_3\in(0,1)$ satisfy
\[
\alpha_1+\alpha_2+\alpha_3=1,
\]
and let $\{r_n\}\subset[r_0,\infty)$ for some $r_0>0$. Choose recurrent control mappings
\[
\tau_T,\tau_A,\tau_\Theta:\mathbb N\to\mathbb N
\]
such that, for every $j,k,l\in\mathbb N$, the sets
\[
\{n\in\mathbb N:\tau_T(n)=j\},\qquad
\{n\in\mathbb N:\tau_A(n)=k\},\qquad
\{n\in\mathbb N:\tau_\Theta(n)=l\}
\]
are infinite. Define
\[
\widehat T_n:=T_{\tau_T(n)},\qquad \widehat A_n:=A_{\tau_A(n)},\qquad \widehat\Theta_n:=\Theta_{\tau_\Theta(n)}.
\]
Given $x_1=x\in C$ and $C_1=C$, define the sequence $\{x_n\}$ by
\begin{equation}\label{alg}
\begin{cases}
z_n=R_{r_n}^{\widehat\Theta_n}x_n,\\
u_n=Q_{r_n}^{\widehat A_n}x_n,\\
y_n=\Bmix(x_n,z_n,J_*\widehat T_nu_n),\\
C_{n+1}=\{z\in C_n:\Dg(z,y_n)\leq\Dg(z,x_n)\},\\
x_{n+1}=\Bproj_{C_{n+1}}x,
\end{cases}
\end{equation}
for all $n\geq1$.

\section{Main results and corollaries}\label{sec:main}

We now state the well-definedness result, the principal strong-convergence theorem, its localized
counterpart, and the immediate Hilbert-space and Alber-functional consequences.
Their proofs are deferred to Section~\ref{sec:proofs}.

\subsection{Convergence results}

\begin{lemma}\label{welldefined}
Let $E$, $C$, $g$, the families
$\{\Theta_l\}_{l\geq1}$, $\{A_k\}_{k\geq1}$, and
$\{T_j\}_{j\geq1}$, the parameters
$\alpha_1,\alpha_2,\alpha_3$, $\{r_n\}$, and the control mappings
$\tau_T$, $\tau_A$, and $\tau_\Theta$ be as specified in
Section~\ref{sec:problem}. Assume that $\Sol\neq\emptyset$ and that
Assumption~\ref{resolventprops} holds for every $\Theta_l$ and every
$A_k$. Then the sequence $\{x_n\}$ generated by \eqref{alg} is well
defined. Moreover,
\[
\Sol\subset C_n,\qquad n\geq1.
\]
\end{lemma}

\begin{theorem}\label{mainthm}
Let $E$ be a uniformly smooth and uniformly convex real Banach space and let $C\subset U$ be a
nonempty closed and convex subset of $E$. Let $g:E\to(-\infty,+\infty]$ satisfy $(G1)-(G5)$. Let
$\{\Theta_l\}_{l\geq1}$ be a countable family of bifunctions from $C\times C$ into $\mathbb R$, each
satisfying $(E1)-(E4)$, let $\{A_k\}_{k\geq1}$ be a countable family of norm-continuous monotone
mappings from $C$ into $E^{*}$, and let $\{T_j\}_{j\geq1}$ be a countable family of $J_*$-closed
generalized Bregman $J_*$-nonexpansive mappings from $C$ into $E^{*}$. Assume that
Assumption~\ref{resolventprops} holds for every $\Theta_l$ and every $A_k$, that $\Sol$ defined by
\eqref{commonset} is nonempty, and that the recurrent controls $\tau_T$, $\tau_A$, and
$\tau_\Theta$ activate every index infinitely often as specified above. Then the sequence $\{x_n\}$ generated
by \eqref{alg} converges strongly to $\Bproj_{\Sol}x$.
\end{theorem}

\begin{corollary}[Localized convergence theorem]\label{localmainthm}
Assume all the hypotheses of Theorem~\ref{mainthm} concerning the space $E$,
the constraint set $C$, the families $\{\Theta_l\}_{l\geq1}$,
$\{A_k\}_{k\geq1}$, and $\{T_j\}_{j\geq1}$, the componentwise resolvent
solvability, the common solution set $\Sol$, the recurrent controls, and the
algorithmic parameters. Replace only the assumptions $(G1)$--$(G5)$ on the
Bregman generator $g$ by the following localized conditions. Suppose that $C$ is bounded and that there are bounded sets $K_0,K_1\subset U$
and a bounded convex set $K^{*}\subset U^{*}$ such that
\begin{itemize}
\item[(L1)] $C\subset K_0$, $\nabla g(K_0)\subset K^{*}$, and
\[
\Bmix(K_0\times K_0\times K_0)\subset K_1,
\]
\item[(L2)] $g$ and $\nabla g$ are continuous on a bounded subset of $U$ containing $K_0\cup K_1$,
$\nabla g$ is uniformly continuous on $K_0$, and $g$ is totally convex along sequences in
$K_0\cup K_1$, in the sense that for any sequences $\{\zeta_n\},\{\eta_n\}\subset K_0\cup K_1$,
\[
\Dg(\zeta_n,\eta_n)\to0\quad\Longrightarrow\quad \|\zeta_n-\eta_n\|\to0;
\]
\item[(L3)] $g^{*}$ is uniformly convex on $K^{*}$, and $\nabla g^{*}$ is bounded and uniformly
continuous on $K^{*}$.
\end{itemize}
Then the sequence generated by \eqref{alg} converges strongly to $\Bproj_{\Sol}x$.
\end{corollary}

\begin{corollary}[Hilbert space case]
Let $E=H$ be a real Hilbert space and let $g(x)=\frac12\|x\|^{2}$. Under the standard Riesz
identification of $H$ with $H^{*}$, one has $J=J_*=I$ and $\nabla g=I$. Hence
$\Dg(x,y)=\frac12\|x-y\|^{2}$, and the algorithm \eqref{alg} becomes
\[
\begin{cases}
z_n=R_{r_n}^{\widehat\Theta_n}x_n,\\
u_n=Q_{r_n}^{\widehat A_n}x_n,\\
y_n=\alpha_1x_n+\alpha_2z_n+\alpha_3T_{\tau_T(n)}u_n,\\
C_{n+1}=\{z\in C_n:\|z-y_n\|\leq\|z-x_n\|\},\\
x_{n+1}=P_{C_{n+1}}x.
\end{cases}
\]
Under the corresponding assumptions of Theorem \ref{mainthm}, $\{x_n\}$ converges strongly to
$P_{\Sol}x$.
\end{corollary}

\begin{corollary}[Alber functional case]
Let $E$ be a uniformly convex and uniformly smooth Banach space, and let
$g(x)=\frac12\|x\|^{2}$. Then $\nabla g=J$ and
\[
\Dg(x,y)=\frac12\left(\|x\|^{2}-2\langle x,Jy\rangle+\|y\|^{2}\right)=\frac12\phi(x,y).
\]
In this case, Definition~\ref{genBregJstar} becomes
\[
\phi(p,J_*Tx)\leq\phi(p,x),
\qquad \forall x\in C,\ p\in F_J(T),
\]
which is the generalized $J_*$-nonexpansive condition. Therefore, Theorem~\ref{mainthm} recovers an
Alber-functional hybrid method for fixed points, equilibrium problems, and variational inequality
problems.
\end{corollary}

\begin{corollary}
Assume that $T_j=T$ for every $j\geq1$, where $T$ is $J_*$-closed and generalized Bregman
$J_*$-nonexpansive. Then the sequence generated by \eqref{alg} converges strongly to the Bregman
projection of $x$ onto
\[
F_J(T)\cap\left(\bigcap_{l=1}^{\infty}\EP(\Theta_l)\right)
\cap\left(\bigcap_{k=1}^{\infty}\VI(C,A_k)\right).
\]
\end{corollary}

\begin{remark}
Theorem~\ref{mainthm} is a Legendre--Bregman counterpart of Alber-functional hybrid schemes such as \cite{UbaCarpathian2023,UbaViscosity2026,UbaCountableMonotone2026}. The proof requires dual-coordinate Bregman mixing, bounded-set properties of $g$ and $g^*$, componentwise recurrent activation, and a localized argument that covers entropy generators. The inertial Bregman framework in \cite{UbaInertial2026} concerns a different fixed-point-only algorithm and does not contain the equilibrium and variational-inequality resolvent structure of the present theorem.
\end{remark}

\section{Preliminary tools and auxiliary results}\label{sec:prelim}

This section collects the geometric, mapping, and resolvent tools used in the outer-approximation analysis. Standard results are explicitly cited, while proofs are included in Section~\ref{sec:proofs} either for completeness or because the exact formulation used here is problem-specific.

\subsection{Bregman projection and dual-coordinate tools}

\begin{lemma}\cite{AlberButnariu1997,BauschkeBorwein1997,ButnariuIusem2000}\label{BprojExistence}
Assume that $E$ is reflexive and that $g$ is a strongly coercive Legendre function. If
$D\subset U$ is nonempty, closed, and convex, then $\Bproj_Dx$ exists and is unique for every
$x\in U$.
\end{lemma}

\begin{lemma}\cite{AlberButnariu1997,BauschkeBorwein1997,ButnariuIusem2000}\label{Bprojlemma}
Let $D\subset U$ be nonempty, closed and convex, and let $z=\Bproj_Dx$. Then
\begin{equation}\label{Bprojchar}
\langle y-z,\nabla g(x)-\nabla g(z)\rangle\leq0,
\qquad \forall y\in D.
\end{equation}
Equivalently,
\begin{equation}\label{Bprojineq}
\Dg(y,z)+\Dg(z,x)\leq \Dg(y,x),
\qquad \forall y\in D.
\end{equation}
\end{lemma}

\begin{lemma}\cite{Bregman1967,ButnariuIusem2000}\label{threepoint}
For all $x,y,z\in U$, the following identity holds:
\[
\Dg(x,z)=\Dg(x,y)+\Dg(y,z)+\langle x-y,\nabla g(y)-\nabla g(z)\rangle.
\]
\end{lemma}

\begin{lemma}\label{convexdual}
Let $x_1,x_2,x_3,p\in U$ and let $\alpha_1,\alpha_2,\alpha_3\in(0,1)$ with
$\alpha_1+\alpha_2+\alpha_3=1$. Set
\[
y=\Bmix(x_1,x_2,x_3).
\]
Then the Bregman mixture $y$ is well defined, belongs to $U$, and
\[
\Dg(p,y)\leq \alpha_1\Dg(p,x_1)+\alpha_2\Dg(p,x_2)+\alpha_3\Dg(p,x_3).
\]
Moreover, let $\mathcal{V}^{*}\subset U^{*}$ be bounded. There is a nondecreasing function
$\rho_{\mathcal{V}^{*}}:[0,\infty)\to[0,\infty)$ such that
\[
\rho_{\mathcal{V}^{*}}(0)=0,
\qquad
\rho_{\mathcal{V}^{*}}(t)>0\quad\hbox{for every }t>0,
\]
and, whenever $\nabla g(x_i)\in \mathcal{V}^{*}$ for $i=1,2,3$,
\begin{align}\label{strictdual}
\Dg(p,y)&\leq \alpha_1\Dg(p,x_1)+\alpha_2\Dg(p,x_2)+\alpha_3\Dg(p,x_3)\nonumber\\
&\quad-\frac{\alpha_1\alpha_3}{\alpha_1+\alpha_3}\rho_{\mathcal{V}^{*}}(\|\nabla g(x_1)-\nabla g(x_3)\|).
\end{align}
The basic Jensen-type estimate follows from the dual representation of a Bregman distance and convexity of $g^*$; compare \cite{BauschkeBorwein1997,ButnariuIusemZalinescu2003}. The explicit three-point modulus coefficient in \eqref{strictdual} is proved below in the form required by the algorithm.
\end{lemma}

\begin{lemma}[Dual residual conversion]\label{dualresidual}
Assume the uniform-continuity clause in $(G5)$. Let $\{s_n\}$ and $\{t_n\}$ be bounded sequences in
$U^{*}$ such that $\|s_n-t_n\|\to0$. Then
\[
\|\nabla g^{*}(s_n)-\nabla g^{*}(t_n)\|\to0.
\]
\end{lemma}

\subsection{Projection-generated mappings and componentwise resolvents}

\begin{proposition}[Projection-generated generalized Bregman $J_*$-nonexpansive mappings]
\label{projectiongenerated}
Let $D\subset C$ be nonempty, closed, and convex. Assume that the Bregman projection
$\Bproj_Dx$ exists for every $x\in C$, and that $\nabla g$ is norm-continuous on $U$.
Define the mapping
\[
T_Dx:=J(\Bproj_Dx),\qquad x\in C.
\]
Then $J_*T_D=\Bproj_D$, and $T_D$ is a $J_*$-closed generalized Bregman
$J_*$-nonexpansive mapping satisfying
\[
F_J(T_D)=D.
\]
\end{proposition}

\begin{proposition}\cite{TakahashiZembayashi2008,CombettesHirstoaga2005,ChenBiSu2015,ReichSabach2010}\label{resolventconsequences}
Let $g$ be Legendre, let $C\subset U$ be nonempty, closed, and convex, and let $r>0$.
\begin{enumerate}
\item If $\Theta:C\times C\to\mathbb R$ satisfies \textup{(E1)}--\textup{(E4)} and
$\mathcal Z_r^{\Theta}(x)$ is nonempty for every $x\in C$, then each such set is a singleton.
Its unique element is denoted by $R_r^{\Theta}x$. Moreover,
\[
\Fix(R_r^{\Theta})=\EP(\Theta),
\]
and
\begin{equation}\label{resineq1}
\Dg(p,R_r^{\Theta}x)+\Dg(R_r^{\Theta}x,x)\leq \Dg(p,x),
\qquad p\in\EP(\Theta),\ x\in C.
\end{equation}
\item If $A:C\to E^{*}$ is monotone and $\mathcal W_r^{A}(x)$ is nonempty for every $x\in C$,
then each such set is a singleton. Its unique element is denoted by $Q_r^{A}x$. Moreover,
\[
\Fix(Q_r^{A})=\VI(C,A),
\]
and
\begin{equation}\label{resineq2}
\Dg(q,Q_r^{A}x)+\Dg(Q_r^{A}x,x)\leq \Dg(q,x),
\qquad q\in\VI(C,A),\ x\in C.
\end{equation}
\end{enumerate}
\end{proposition}

\begin{proposition}\cite{ChenBiSu2015,ReichSabach2010}\label{standardresolvent}
Assume that $g:E\to\mathbb R$ is a continuous, strongly coercive Legendre function that is bounded
and uniformly convex on bounded subsets of $E$. Let $C\subset E$ be nonempty, closed, and convex.
If $\Theta:C\times C\to\mathbb R$ satisfies \textup{(E1)}--\textup{(E4)}, then
$\mathcal Z_r^{\Theta}(x)$ is nonempty for every $x\in C$ and $r>0$. If, in addition,
$A:C\to E^*$ is norm continuous and monotone, then $\mathcal W_r^A(x)$ is nonempty for every
$x\in C$ and $r>0$. Consequently, Proposition~\ref{resolventconsequences} supplies the resolvent
mappings, their fixed-point identities, and the estimates \eqref{resineq1}--\eqref{resineq2}.
\end{proposition}

\begin{proposition}[Potential-generated localized resolvents]\label{potentialresolvent}
Let $E=\mathbb R^d$, let $C\subset U$ be nonempty, compact, and convex, and suppose that $g$ is
$C^1$ and strictly convex on an open neighborhood of $C$. Let $\Phi$ be a convex function that is
$C^1$ on an open neighborhood of $C$, and define
\[
A_{\Phi}(z):=\nabla\Phi(z),
\qquad
\Theta_{\Phi}(z,y):=\langle\nabla\Phi(z),y-z\rangle.
\]
Then $A_{\Phi}$ is continuous and monotone, $\Theta_{\Phi}$ satisfies \textup{(E1)}--\textup{(E4)},
and, for every $x\in C$ and $r>0$,
\[
\mathcal Z_r^{\Theta_{\Phi}}(x)
=
\mathcal W_r^{A_{\Phi}}(x)
=
\left\{\argmin_{w\in C}\left(r\Phi(w)+\Dg(w,x)\right)\right\}.
\]
In particular, both candidate sets are nonempty, and all conclusions of
Proposition~\ref{resolventconsequences} hold.
The minimization characterization is the standard first-order optimality condition for a differentiable convex potential; see \cite{Rockafellar1970CA}.
\end{proposition}

\subsection{Solution-set structure}

\begin{proposition}[Structure of the individual solution sets]\label{setstructure}
Let $T:C\to E^*$ be a $J_*$-closed generalized Bregman $J_*$-nonexpansive mapping, let
$\Theta:C\times C\to\mathbb R$ satisfy $(E1)$--$(E4)$, and let $A:C\to E^*$ be norm continuous
and monotone. Then $F_J(T)$, $\EP(\Theta)$, and $\VI(C,A)$ are closed and convex subsets of $C$.
For the equilibrium and variational-inequality assertions; see \cite{BlumOettli1994,CombettesHirstoaga2005,TakahashiZembayashi2008,Minty1962}. The $J$-fixed-point assertion is verified directly below from Definition~\ref{genBregJstar} and $J_*$-closedness.
\end{proposition}

\section{Examples}\label{sec:examples}

\subsection{Negative-entropy setting}

Throughout this subsection, let $E=\mathbb R^{m}$ with its Euclidean norm, fix
$\delta\in(0,1/m)$, and set
\[
C=\left\{x=(x_1,\ldots,x_m)\in\mathbb R^{m}:x_i\geq\delta,\ \sum_{i=1}^{m}x_i=1\right\}.
\]
We use the negative-entropy generator, with the convention $0\log 0:=0$,
\[
g(x)=\sum_{i=1}^{m}x_i\log x_i\quad (x\in\mathbb R^m_+),
\qquad
g(x)=+\infty\quad (x\notin\mathbb R^m_+).
\]
Its interior domain is $U=\mathbb R^m_{++}$ and
\[
\Dg(x,y)=\sum_{i=1}^{m}\left[x_i\log\!\left(\frac{x_i}{y_i}\right)-x_i+y_i\right],
\qquad x\in\dom g,\ y\in U.
\]
When $x,y\in C$, the linear terms cancel, but they need not cancel when the second argument is an
auxiliary Bregman average. Under the Euclidean Riesz identification, $E^*=E$ and $J=J_*=I$;
this fact comes from the norm and is independent of the entropy generator. In contrast,
\[
\nabla g(x)=(1+\log x_i)_{i=1}^{m},
\qquad
\nabla g^{*}(s)=(e^{s_i-1})_{i=1}^{m}.
\]

Let
\[
K_0:=K_1:=K:=[\delta,1]^m\subset U,
\qquad
K^{*}:=[1+\log\delta,1]^m\subset U^{*}.
\]
Then $C\subset K_0$, the set $K^*$ is convex, and
\[
\nabla g(K_0)=K^{*},
\qquad
\nabla g^{*}(K^{*})=K_0.
\]
Moreover, if $v_1,v_2,v_3\in K_0$, then every coordinate of the Bregman mixture
$\Bmix(v_1,v_2,v_3)$ is a weighted geometric mean of three numbers in $[\delta,1]$. Hence
\[
\Bmix(K_0\times K_0\times K_0)\subset K_1.
\]
Thus the set-invariance part of condition \textup{(L1)} in Corollary~\ref{localmainthm} is verified
explicitly. On a neighborhood of $K$, the functions $g$ and $\nabla g$ are smooth; moreover,
\[
\nabla^2g(x)=\operatorname{diag}(x_i^{-1}),
\qquad
\nabla^2g^*(s)=\operatorname{diag}(e^{s_i-1}).
\]
Since $x_i\leq1$ on $K$ and $e^{s_i-1}\geq\delta$ on $K^*$, these Hessians satisfy
\[
\nabla^2g(x)\succeq I \quad (x\in K),
\qquad
\nabla^2g^*(s)\succeq \delta I \quad (s\in K^*).
\]
Thus $g$ is uniformly convex, and hence totally convex, on $K_0\cup K_1=K$, while $g^*$ is
uniformly convex on $K^*$. Since $\nabla g$ and $\nabla g^*$ are smooth on neighborhoods of the
compact sets $K_0$ and $K^*$, respectively, both maps are uniformly continuous there; moreover,
$\nabla g^*$ is bounded on $K^*$. Consequently, conditions \textup{(L1)}--\textup{(L3)} of
Corollary~\ref{localmainthm} hold with the explicitly chosen sets $K_0$, $K_1$, and $K^*$ once the
componentwise constraints are verified.

\begin{example}[A singleton feasibility instance]\label{singletonentropy}
Choose $p:=(1/m,\ldots,1/m)\in C$, fix numbers $\overline\lambda\in(0,1)$ and
\[
0\leq\lambda_j\leq\overline\lambda<1,
\qquad a_k>0,
\qquad b_l>0,
\qquad j,k,l\in\mathbb N,
\]
and define the primal affine map and its representative by
\[
S_jx:=(1-\lambda_j)p+\lambda_jx,
\qquad
T_jx:=J(S_jx),
\]
together with
\[
A_kx:=a_k(x-p),
\qquad
\Theta_l(x,y):=b_l\langle x-p,y-x\rangle.
\]
Under the Euclidean Riesz identification, $J=J_*=I$, so $J_*T_j=S_j$. Each $S_j$ is
continuous, maps $C$ into $C$, and satisfies $F_J(T_j)=\{p\}$. For the entropy generator,
\[
\nabla_y^2\Dg(p,y)=\operatorname{diag}\left(\frac{p_i}{y_i^2}\right)\succ0
\qquad (y\in U),
\]
so $y\mapsto\Dg(p,y)$ is convex on $U$. Jensen's inequality therefore gives
\[
\Dg(p,J_*T_jx)=\Dg(p,S_jx)\leq\lambda_j\Dg(p,x)\leq\Dg(p,x).
\]
Thus every $T_j$ is generalized Bregman $J_*$-nonexpansive and $J_*$-closed. The bifunctions satisfy
$(E1)$--$(E4)$, with
\[
\Theta_l(x,y)+\Theta_l(y,x)=-b_l\|x-y\|^2\leq0.
\]
For this particular reference point $p$, testing the defining equilibrium or
variational-inequality condition with $y=p$ yields
\[
\EP(\Theta_l)=\{p\},
\qquad
\VI(C,A_k)=\{p\}.
\]

For $c,r>0$ and $x\in C$, let
\[
\Phi_{c,r,x}(w):=\frac c2\|w-p\|^2+\frac1r\Dg(w,x),
\qquad w\in C.
\]
This is the potential-generated situation of Proposition~\ref{potentialresolvent}, with
$\Phi(w)=\frac c2\|w-p\|^2$. Hence the two candidate resolvent sets are nonempty singletons, their
resolvent mappings coincide with the displayed minimizer, and the Bregman--Fej\'er estimates follow
from Proposition~\ref{resolventconsequences}. Therefore $\Sol=\{p\}$. The Bregman average has coordinates
\[
(y_n)_i=(x_n)_i^{\alpha_1}(z_n)_i^{\alpha_2}
\left((1-\lambda_{\tau_T(n)})p_i+\lambda_{\tau_T(n)}(u_n)_i\right)^{\alpha_3},
\]
which belong to $[\delta,1]$. Thus Example~\ref{singletonentropy} satisfies the hypotheses of the
localized convergence result, Corollary~\ref{localmainthm}, which yields $x_n\to p$.
\end{example}

\begin{example}[A non-singleton common solution set with distinct countable constraints]
\label{nonsingleentropy}
Let $m=3$, retain the preceding truncated simplex $C$, and write $d:=e_1-e_2$ and
$p:=(1/3,1/3,1/3)$. For $q\in\mathbb N$, set
\[
\varrho_q:=\frac{1-3\delta}{q+1}.
\]
Then $\varrho_q\downarrow0$ and $0<\varrho_q<1-3\delta$. Define the pairwise distinct
closed convex subsets
\[
D_{2q-1}:=\left\{x\in C:x_1-x_2\leq\varrho_q\right\},
\qquad
D_{2q}:=\left\{x\in C:x_2-x_1\leq\varrho_q\right\}.
\]
Set
\[
T_j:=J\circ\Bproj_{D_j},
\qquad
A_kx:=a_k(x_1-x_2)d,
\qquad
\Theta_l(x,y):=b_l(x_1-x_2)\bigl[(y_1-y_2)-(x_1-x_2)\bigr],
\]
where $a_k,b_l>0$. Under the Euclidean Riesz identification, $J=J_*=I$ and hence
$J_*T_j=\Bproj_{D_j}$. Proposition~\ref{projectiongenerated} shows that every $T_j$ is a
$J_*$-closed generalized Bregman $J_*$-nonexpansive mapping and that
\[
F_J(T_j)=D_j.
\]

The strict decrease of $\{\varrho_q\}$ makes the two one-sided subfamilies strictly nested,
and the positive and negative one-sided constraints are different. Thus the family $\{D_j\}_{j\geq1}$
has pairwise distinct members. Its fixed-point constraints have the non-singleton intersection
\[
\bigcap_{j=1}^{\infty}D_j
=\mathcal{F}:=\{x\in C:x_1=x_2\}
=\left\{(s,s,1-2s):\delta\leq s\leq\frac{1-\delta}{2}\right\}.
\]
This is a nondegenerate line segment because $\delta<1/3$. The mappings $A_k$ are norm-continuous
and monotone, since
\[
\langle x-y,A_kx-A_ky\rangle=a_k\bigl((x_1-x_2)-(y_1-y_2)\bigr)^2\geq0.
\]
The bifunctions satisfy $(E1)$--$(E4)$, and
\[
\Theta_l(x,y)+\Theta_l(y,x)
=-b_l\bigl((x_1-x_2)-(y_1-y_2)\bigr)^2\leq0.
\]
For the particular point $p\in \mathcal{F}$, testing the defining condition with $y=p$
shows the reverse inclusions
\[
\VI(C,A_k)\subset \mathcal{F},
\qquad
\EP(\Theta_l)\subset \mathcal{F}.
\]
Together with the immediate inclusions $\mathcal{F}\subset\VI(C,A_k)$ and
$\mathcal{F}\subset\EP(\Theta_l)$, this gives
\[
\VI(C,A_k)=\mathcal{F},
\qquad
\EP(\Theta_l)=\mathcal{F}
\]
for every $k,l$. Indeed, the two tests yield, respectively,
$-a_k(x_1-x_2)^2\geq0$ and $-b_l(x_1-x_2)^2\geq0$.

For $c,r>0$ and $x\in C$, define
\[
\Psi_{c,r,x}(w):=\frac c2(w_1-w_2)^2+\frac1r\Dg(w,x),
\qquad w\in C.
\]
This strictly convex continuous functional has a unique minimizer. Its optimality condition is exactly
the variational-inequality resolvent condition for $A^{(c)}w=c(w_1-w_2)d$ and, equivalently, the
equilibrium-resolvent condition for
\[
\Theta^{(c)}(w,y)=c(w_1-w_2)\bigl[(y_1-y_2)-(w_1-w_2)\bigr].
\]
This is again covered by Proposition~\ref{potentialresolvent}, now with
$\Phi(w)=\frac c2(w_1-w_2)^2$. Consequently every componentwise candidate resolvent set is a
nonempty singleton and Proposition~\ref{resolventconsequences} supplies the fixed-point identities
and Bregman--Fej\'er estimates. Therefore
\[
\Sol=\mathcal{F}.
\]
Because $J_*T_j(C)=\Bproj_{D_j}(C)\subset C$ and both resolvents map $C$ into $C$, all three
primal arguments in each Bregman average belong to $C$. Their coordinatewise weighted geometric mean
belongs to $K$, so the hypotheses of the localized convergence result, Corollary~\ref{localmainthm}, hold. Consequently, for
every $x\in C$, the algorithm converges strongly to the entropy-Bregman projection $\Bproj_{\mathcal{F}}x$. This
example contains countably many distinct fixed-point constraints and a non-singleton common solution
set.
\end{example}

\section{Proofs}\label{sec:proofs}

\subsection{Proofs of the preliminary results}

We first prove the preliminary statements collected in Section~\ref{sec:prelim}. They provide the Bregman
projection properties, resolvent estimates, and closed-convexity facts used in the outer-approximation
argument.

\noindent\textbf{Proof of Lemma~\ref{BprojExistence}.}
Fix $x\in U$. Since
\[
\Dg(z,x)=g(z)-\langle\nabla g(x),z\rangle
+\bigl\langle\nabla g(x),x\bigr\rangle-g(x),
\]
strong coercivity of $g$ implies that $z\mapsto\Dg(z,x)$ is coercive on $D$. Let $\{z_n\}\subset D$
be a minimizing sequence. It is bounded, and reflexivity yields a weakly convergent subsequence,
still denoted by $\{z_n\}$, with $z_n\rightharpoonup z$. Since $D$ is norm closed and convex, it is
weakly closed; hence $z\in D$. The convex lower semicontinuity of $g$ implies that
$z\mapsto\Dg(z,x)$ is weakly lower semicontinuous. Therefore $z$ minimizes $\Dg(\cdot,x)$ on $D$.
Finally, the essential strict convexity of the Legendre function $g$ makes
$\Dg(\cdot,x)$ strictly convex on $D\subset U$, so the minimizer is unique. This is the standard
direct-method proof of the Bregman projection theorem; see also \cite{BauschkeBorwein1997}.

\noindent\hfill$\square$

\noindent\textbf{Proof of Lemma~\ref{convexdual}.}
Put $s_i:=\nabla g(x_i)$ and $\bar s:=\alpha_1s_1+\alpha_2s_2+\alpha_3s_3$. Since $U^{*}$ is convex,
$\bar s\in U^{*}$ and $y=\nabla g^{*}(\bar s)\in U$. The Fenchel identity gives
\[
\Dg(p,\nabla g^{*}(s))=g(p)+g^{*}(s)-\langle p,s\rangle,
\qquad s\in U^{*}.
\]
Consequently, the first inequality follows directly from the convexity of $g^{*}$.

For the strict estimate, choose a bounded convex set $\widehat{\mathcal{V}}^{*}\subset U^{*}$ containing
$\mathcal{V}^{*}$. By the uniform convexity of $g^{*}$ on $\widehat{\mathcal{V}}^{*}$, there is a modulus
$\rho_{\mathcal{V}^{*}}$ with the stated positivity property such that, for $s,t\in \mathcal{V}^{*}$ and
$\lambda\in[0,1]$,
\[
g^{*}(\lambda s+(1-\lambda)t)
\leq \lambda g^{*}(s)+(1-\lambda)g^{*}(t)
-\lambda(1-\lambda)\rho_{\mathcal{V}^{*}}(\|s-t\|).
\]
Apply this inequality first to $s_1$ and $s_3$ with
$\lambda=\alpha_1/(\alpha_1+\alpha_3)$, and then use ordinary convexity to incorporate $s_2$.
Since $\alpha_1+\alpha_3=1-\alpha_2$, this yields the exact coefficient
\[
g^{*}(\bar s)
\leq \sum_{i=1}^{3}\alpha_i g^{*}(s_i)
-\frac{\alpha_1\alpha_3}{\alpha_1+\alpha_3}\rho_{\mathcal{V}^{*}}(\|s_1-s_3\|).
\]
(In particular, because $\alpha_1+\alpha_3\leq1$, this is stronger than the
corresponding estimate with coefficient $\alpha_1\alpha_3$.)
Adding $g(p)-\langle p,\bar s\rangle$ proves \eqref{strictdual}.

\noindent\hfill$\square$

\noindent\textbf{Proof of Lemma~\ref{dualresidual}.}
The union of the two sequences is a bounded subset of $U^{*}$. The conclusion follows directly from
uniform continuity of $\nabla g^{*}$ on that set.

\noindent\hfill$\square$

\noindent\textbf{Proof of Proposition~\ref{projectiongenerated}.}
The identity $J_*T_D=\Bproj_D$ follows from $J_*=J^{-1}$. If $p\in D$, then
$\Bproj_Dp=p$, because $\Dg(p,p)=0$ is the unique minimum over $D$; hence
$T_Dp=Jp$. Conversely, if $T_Dp=Jp$, then $\Bproj_Dp=p$, and since
$\Bproj_Dp\in D$, we obtain $p\in D$. Thus $F_J(T_D)=D$.

For $p\in D$ and $x\in C$, the Bregman projection inequality gives
\[
\Dg(p,J_*T_Dx)=\Dg(p,\Bproj_Dx)\leq \Dg(p,x),
\]
so $T_D$ is generalized Bregman $J_*$-nonexpansive.

It remains to prove $J_*$-closedness. Let $x_n\to x$ in $C$ and suppose that
$\Bproj_Dx_n\to y$. Since $D$ is closed, $y\in D$. For every $v\in D$, the
projection variational inequality gives
\[
\langle v-\Bproj_Dx_n,\nabla g(x_n)-\nabla g(\Bproj_Dx_n)\rangle\leq0.
\]
The continuity of $\nabla g$ permits passage to the limit, yielding
\[
\langle v-y,\nabla g(x)-\nabla g(y)\rangle\leq0,
\qquad v\in D.
\]
By the three-point identity, this implies
\[
\Dg(v,x)\geq \Dg(y,x),\qquad v\in D.
\]
The uniqueness of the Bregman projection therefore gives $y=\Bproj_Dx$. Hence
$J_*T_D=\Bproj_D$ is closed, which proves the claim.

\noindent\hfill$\square$

\noindent\textbf{Proof of Proposition~\ref{resolventconsequences}.}
We first consider the equilibrium candidate set. Let $z_1,z_2\in\mathcal Z_r^{\Theta}(x)$.
Testing the defining inequality for $z_1$ at $y=z_2$ and the one for $z_2$ at $y=z_1$, then adding,
gives
\[
0\leq
\Theta(z_1,z_2)+\Theta(z_2,z_1)
-\frac1r\langle z_1-z_2,\nabla g(z_1)-\nabla g(z_2)\rangle.
\]
The first two terms are nonpositive by \textup{(E2)}, whereas
\[
\langle z_1-z_2,\nabla g(z_1)-\nabla g(z_2)\rangle
=\Dg(z_1,z_2)+\Dg(z_2,z_1)\geq0.
\]
Thus both Bregman distances vanish, and the strict convexity associated with a Legendre generator
implies $z_1=z_2$.

If $w\in\EP(\Theta)$, then $w\in\mathcal Z_r^{\Theta}(w)$, and hence $R_r^{\Theta}w=w$.
Conversely, if $R_r^{\Theta}w=w$, then the defining inequality of
$\mathcal Z_r^{\Theta}(w)$ reduces to $\Theta(w,y)\geq0$ for all $y\in C$.
This proves the fixed-point identity. Now let $p\in\EP(\Theta)$ and put $z=R_r^{\Theta}x$.
Using $y=p$ in the defining inequality gives
\[
\Theta(z,p)+\frac1r\langle p-z,\nabla g(z)-\nabla g(x)\rangle\geq0.
\]
Since $\Theta(p,z)\geq0$ and \textup{(E2)} holds, $\Theta(z,p)\leq0$. Hence
\[
\langle p-z,\nabla g(z)-\nabla g(x)\rangle\geq0.
\]
The three-point identity yields \eqref{resineq1}.

For the variational inequality candidate set, let $z_1,z_2\in\mathcal W_r^A(x)$.
The analogous two tests give
\[
0\leq
-\langle z_1-z_2,Az_1-Az_2\rangle
-\frac1r\langle z_1-z_2,\nabla g(z_1)-\nabla g(z_2)\rangle.
\]
Monotonicity of $A$ and the preceding Bregman identity again force $z_1=z_2$.
The fixed-point identity follows immediately by setting $x=z$ in \eqref{VIresolvent}.
Finally, let $q\in\VI(C,A)$ and put $z=Q_r^Ax$. Monotonicity gives
\[
\langle z-q,Az\rangle
=\langle z-q,Aq\rangle+\langle z-q,Az-Aq\rangle\geq0.
\]
Taking $y=q$ in \eqref{VIresolvent} therefore yields
\[
\langle q-z,\nabla g(z)-\nabla g(x)\rangle\geq0,
\]
and the three-point identity gives \eqref{resineq2}.

\noindent\hfill$\square$

\noindent\textbf{Proof of Proposition~\ref{standardresolvent}.}
For the equilibrium candidate set, solvability is the full-domain Bregman-resolvent theorem of
Chen--Bi--Su \cite[Lemmas~2.13--2.14]{ChenBiSu2015}; see also the related Bregman resolvent framework
in Reich and Sabach \cite{ReichSabach2010}. For the variational inequality candidate set, use the
same result for the monotone perturbation $G_A(z,y)=\langle Az,y-z\rangle$, as in
Chen--Bi--Su \cite[Lemma~2.15]{ChenBiSu2015}. Proposition~\ref{resolventconsequences} then applies.

\noindent\hfill$\square$

\noindent\textbf{Proof of Proposition~\ref{potentialresolvent}.}
The function $w\mapsto r\Phi(w)+\Dg(w,x)$ is continuous and strictly convex on the compact convex
set $C$. It therefore has a unique minimizer, say $z$. Its variational optimality condition is
\[
\langle y-z,r\nabla\Phi(z)+\nabla g(z)-\nabla g(x)\rangle\geq0,
\qquad y\in C.
\]
After division by $r$, this is both the condition defining $\mathcal W_r^{A_{\Phi}}(x)$ and the one
defining $\mathcal Z_r^{\Theta_{\Phi}}(x)$. The remaining assertions follow because the gradient of a
convex $C^1$ function is monotone. The map $\Theta_{\Phi}$ is continuous in its first variable,
so \textup{(E3)} holds, and $\Theta_{\Phi}(z,\cdot)$ is affine and therefore convex and lower
semicontinuous. Finally,
\[
\Theta_{\Phi}(z,y)+\Theta_{\Phi}(y,z)
=-\langle z-y,\nabla\Phi(z)-\nabla\Phi(y)\rangle\leq0.
\]

\noindent\hfill$\square$

\noindent\textbf{Proof of Proposition~\ref{setstructure}.}
To prove the assertion for $F_J(T)$, let $p,q\in F_J(T)$ and set
$w_t:=tp+(1-t)q$ for some $t\in[0,1]$. Generalized Bregman $J_*$-nonexpansivity gives
\[
t\Dg(p,J_*T w_t)+(1-t)\Dg(q,J_*T w_t)
\leq t\Dg(p,w_t)+(1-t)\Dg(q,w_t).
\]
The elementary barycenter identity
\[
t\Dg(p,w)+(1-t)\Dg(q,w)
=t\Dg(p,w_t)+(1-t)\Dg(q,w_t)+\Dg(w_t,w)
\]
shows that $\Dg(w_t,J_*T w_t)=0$, and therefore $J_*T w_t=w_t$. Applying $J$ gives
$T w_t=Jw_t$, so $w_t\in F_J(T)$. Thus $F_J(T)$ is convex. If $p_n\in F_J(T)$ and $p_n\to p$, then $J_*Tp_n=p_n\to p$; $J_*$-closedness yields $J_*Tp=p$, and hence $Tp=Jp$. Thus $F_J(T)$ is closed.

Under $(E1)$--$(E4)$, the equilibrium solution set $\EP(\Theta)$ is closed and convex; this is the
standard equilibrium-set theorem, see, for example, \cite{TakahashiZembayashi2008,CombettesHirstoaga2005}.

For the variational inequality set, monotonicity and norm continuity give the Minty characterization
\[
\VI(C,A)=\bigl\{x\in C:\ \langle y-x,Ay\rangle\geq0\ \text{for every }y\in C\bigr\}.
\]
Indeed, one implication follows directly from monotonicity. For the converse, apply the displayed
inequality at $x+t(y-x)$ and let $t\downarrow0$, using norm continuity of $A$. The right-hand side
is an intersection of closed affine half-spaces in $x$, so $\VI(C,A)$ is closed and convex.

\noindent\hfill$\square$

\subsection{Well-definedness and the outer approximation}

\noindent\textbf{Proof of Lemma~\ref{welldefined}.}
For each $n$, the points $x_n$, $z_n$, and $J_*\widehat T_nu_n$ belong to $C\subset U$. Hence their gradients
belong to the convex set $U^{*}=\intdom g^{*}$. The convex combination in \eqref{alg} therefore lies
in $U^{*}$, so $y_n=\nabla g^{*}(\cdot)$ is well defined and belongs to $U$.

Clearly, $C_1=C$ is nonempty, closed, and convex. Suppose that $C_n$ is nonempty, closed, and convex
and that $\Sol\subset C_n$. By the definition of $\Dg$, the inequality
$\Dg(z,y_n)\leq \Dg(z,x_n)$ is equivalent to
\[
\langle z,\nabla g(x_n)-\nabla g(y_n)\rangle
\leq g(y_n)-g(x_n)+\langle x_n,\nabla g(x_n)\rangle-
\langle y_n,\nabla g(y_n)\rangle,
\]
which is affine in $z$. Thus $C_{n+1}$ is closed and convex. Let $p\in \Sol$. Since
$p\in\bigcap_{j\geq1}F_J(T_j)$, the point $p$ belongs to $F_J(\widehat T_n)$. Using
Lemma~\ref{convexdual}, Proposition~\ref{resolventconsequences}, and the generalized Bregman
$J_*$-nonexpansivity of $\widehat T_n$, we obtain
\begin{align*}
\Dg(p,y_n)
&\leq \alpha_1\Dg(p,x_n)+\alpha_2\Dg(p,z_n)
 +\alpha_3\Dg(p,J_*\widehat T_nu_n)\\
&\leq \alpha_1\Dg(p,x_n)+\alpha_2\Dg(p,z_n)+\alpha_3\Dg(p,u_n)\\
&\leq \Dg(p,x_n).
\end{align*}
Hence $p\in C_{n+1}$; in particular, $C_{n+1}$ is nonempty. The Bregman projection
$\Bproj_{C_{n+1}}x$ therefore exists and is unique. Induction completes the proof.

\noindent\hfill$\square$

\subsection{Proof of the main convergence theorem}

\noindent\textbf{Proof of Theorem~\ref{mainthm}.}
By Lemma~\ref{welldefined}, every iterate is well defined and $\Sol\subset C_n$ for all $n\geq1$.
The proof is divided into five steps.

\noindent\textbf{Step 1.} We show that $\{x_n\}$ converges strongly to some point $x^{*}\in C$.

Since $x_n=\Bproj_{C_n}x$ and $\Sol\subset C_n$, Lemma \ref{Bprojlemma} gives
\[
\Dg(x_n,x)\leq \Dg(p,x),\qquad p\in \Sol.
\]
Thus $\{\Dg(x_n,x)\}$ is bounded. By the forward-sublevel part of $(G5)$, the sequence $\{x_n\}$ is bounded. Also,
$x_{n+1}\in C_{n+1}\subset C_n$ and $x_n=\Bproj_{C_n}x$, so
\[
\Dg(x_n,x)\leq \Dg(x_{n+1},x),
\]
which implies that $\{\Dg(x_n,x)\}$ is nondecreasing and bounded. Therefore,
$\lim_{n\to\infty}\Dg(x_n,x)$ exists.

For $m>n$, since $x_m\in C_m\subset C_n$, Lemma \ref{Bprojlemma} yields
\begin{equation}\label{cauchyD}
\Dg(x_m,x_n)\leq \Dg(x_m,x)-\Dg(x_n,x)\to0,
\qquad m,n\to\infty.
\end{equation}
By total convexity of $g$ on bounded subsets of $U$, we have $\|x_m-x_n\|\to0$. Hence $\{x_n\}$ is a
Cauchy sequence. Since $C$ is closed, there exists $x^{*}\in C$ such that
\[
x_n\to x^{*}\qquad \text{as }n\to\infty.
\]

\noindent\textbf{Step 2.} We show that $z_n\to x^{*}$, $u_n\to x^{*}$ and $y_n\to x^{*}$.

Fix $p\in \Sol$. Since $x_n\to x^{*}$, the continuity of $\Dg(p,\cdot)$ implies that
$\{\Dg(p,x_n)\}$ is bounded. The resolvent inequalities and generalized Bregman
$J_*$-nonexpansivity give
\[
\Dg(p,z_n)\leq \Dg(p,x_n),\qquad
\Dg(p,u_n)\leq \Dg(p,x_n),\qquad
\Dg(p,J_*\widehat T_nu_n)\leq \Dg(p,u_n).
\]
The reverse-sublevel part of $(G5)$ therefore bounds $\{z_n\}$, $\{u_n\}$, and
$\{J_*\widehat T_nu_n\}$. By the gradient-boundedness clause in $(G5)$, the four dual sequences
\[
\{\nabla g(x_n)\},\qquad \{\nabla g(z_n)\},\qquad
\{\nabla g(u_n)\},\qquad \{\nabla g(J_*\widehat T_nu_n)\}
\]
are bounded in $E^{*}$. Hence
\[
\xi_n:=\alpha_1\nabla g(x_n)+\alpha_2\nabla g(z_n)+\alpha_3\nabla g(J_*\widehat T_nu_n)
\]
is a bounded sequence in $U^{*}$. The final clause of $(G5)$ gives the boundedness of
$y_n=\nabla g^{*}(\xi_n)$.

Since $x_{n+1}\in C_{n+1}$, we have
\[
\Dg(x_{n+1},y_n)\leq \Dg(x_{n+1},x_n).
\]
By \eqref{cauchyD}, $\Dg(x_{n+1},x_n)\to0$. Since both sequences are now known to be bounded,
total convexity yields
\[
y_n\to x^{*}.
\]
By Proposition~\ref{resolventconsequences},
\begin{equation}\label{uDineq}
\Dg(u_n,x_n)\leq \Dg(p,x_n)-\Dg(p,u_n),
\end{equation}
and
\begin{equation}\label{zDineq}
\Dg(z_n,x_n)\leq \Dg(p,x_n)-\Dg(p,z_n).
\end{equation}
Moreover, the outer-approximation estimate in the proof of Lemma~\ref{welldefined} gives
\[
\Dg(p,y_n)\leq \alpha_1\Dg(p,x_n)+\alpha_2\Dg(p,z_n)+\alpha_3\Dg(p,u_n)
\leq \Dg(p,x_n).
\]
Since $x_n\to x^{*}$ and $y_n\to x^{*}$, the continuity of $y\mapsto\Dg(p,y)$ on $U$ gives
\[
\Dg(p,x_n)-\Dg(p,y_n)\to0.
\]
The resolvent inequalities make both differences
$\Dg(p,x_n)-\Dg(p,z_n)$ and $\Dg(p,x_n)-\Dg(p,u_n)$ nonnegative, while the preceding
convex-combination estimate implies
\[
\begin{aligned}
0&\leq \alpha_2\bigl[\Dg(p,x_n)-\Dg(p,z_n)\bigr]
   +\alpha_3\bigl[\Dg(p,x_n)-\Dg(p,u_n)\bigr]\\
&\leq \Dg(p,x_n)-\Dg(p,y_n)\to0.
\end{aligned}
\]
Since $\alpha_2,\alpha_3>0$, it follows that
\[
\Dg(p,x_n)-\Dg(p,u_n)\to0,
\qquad
\Dg(p,x_n)-\Dg(p,z_n)\to0.
\]
By \eqref{uDineq} and \eqref{zDineq},
\[
\Dg(u_n,x_n)\to0,
\qquad
\Dg(z_n,x_n)\to0.
\]
Using total convexity again, we obtain
\[
u_n\to x^{*},\qquad z_n\to x^{*}.
\]

\noindent\textbf{Step 3.} We prove that $x^{*}\in\bigcap_{k=1}^{\infty}\VI(C,A_k)$.

Since $u_n=Q_{r_n}^{\widehat A_n}x_n$, we have
\begin{equation}\label{VIstep}
\langle y-u_n,\widehat A_nu_n\rangle+\frac1{r_n}\langle y-u_n,\nabla g(u_n)-\nabla g(x_n)\rangle\geq0,
\qquad \forall y\in C.
\end{equation}
From $u_n\to x^{*}$ and $x_n\to x^{*}$, and since $\nabla g$ is uniformly continuous on bounded
subsets of $U$ by $(G1)$--$(G2)$, we get
\[
\frac{\|\nabla g(u_n)-\nabla g(x_n)\|}{r_n}\to0.
\]
Fix $k\in\mathbb N$. The control condition for $\tau_A$ guarantees an infinite subsequence
$\{n_j^{(k)}\}$ such that $\widehat A_{n_j^{(k)}}=A_k$ for every $j$. Let $v\in C$. Monotonicity and
\eqref{VIstep}, with $y=v$, give
\[
\begin{aligned}
\langle v-u_{n_j^{(k)}},A_kv\rangle
&=\langle v-u_{n_j^{(k)}},A_ku_{n_j^{(k)}}\rangle
 +\langle v-u_{n_j^{(k)}},A_kv-A_ku_{n_j^{(k)}}\rangle\\
&\geq \langle v-u_{n_j^{(k)}},A_ku_{n_j^{(k)}}\rangle\\
&\geq -\frac1{r_{n_j^{(k)}}}\langle v-u_{n_j^{(k)}},
\nabla g(u_{n_j^{(k)}})-\nabla g(x_{n_j^{(k)}})\rangle.
\end{aligned}
\]
The sequence $\{v-u_{n_j^{(k)}}\}$ is bounded, whereas the norm of the final gradient difference
divided by $r_{n_j^{(k)}}$ tends to zero. Passing to the limit therefore gives
\[
\langle v-x^{*},A_kv\rangle\geq0,
\qquad \forall v\in C.
\]
For $y\in C$ and $t\in(0,1]$, set $v_t:=x^{*}+t(y-x^{*})\in C$. Then
\[
0\leq \langle v_t-x^{*},A_kv_t\rangle
=t\langle y-x^{*},A_kv_t\rangle.
\]
By the assumed norm continuity of $A_k$, we have $A_kv_t\to A_kx^{*}$ in $E^{*}$ as
$t\downarrow0$. The duality pairing with the fixed vector $y-x^{*}$ is continuous; hence
\[
\langle y-x^{*},A_kx^{*}\rangle\geq0,
\qquad \forall y\in C.
\]
Thus $x^{*}\in\VI(C,A_k)$. Since $k$ was arbitrary,
\[
x^{*}\in\bigcap_{k=1}^{\infty}\VI(C,A_k).
\]

\noindent\textbf{Step 4.} We prove that $x^{*}\in\bigcap_{j=1}^{\infty}F_J(T_j)$.

Let $\mathcal{V}^{*}\subset U^{*}$ be a bounded convex set containing the four gradient sequences
\[
\{\nabla g(x_n)\},\qquad \{\nabla g(z_n)\},\qquad
\{\nabla g(u_n)\},\qquad \{\nabla g(J_*\widehat T_nu_n)\}.
\]
The existence of such a set follows from the boundedness established in Step~2. Using the strict form \eqref{strictdual} of Lemma~\ref{convexdual}, the outer-approximation
estimate, and Proposition~\ref{resolventconsequences}, we obtain
\[
\Dg(p,y_n)\leq \Dg(p,x_n)-\frac{\alpha_1\alpha_3}{\alpha_1+\alpha_3}
\rho_{\mathcal{V}^{*}}(\|\nabla g(x_n)-\nabla g(J_*\widehat T_nu_n)\|).
\]
Since $\Dg(p,x_n)-\Dg(p,y_n)\to0$, the positivity property of $\rho_{\mathcal{V}^{*}}$ implies
\begin{equation}\label{residual1}
\|\nabla g(x_n)-\nabla g(J_*\widehat T_nu_n)\|\to0.
\end{equation}
Since $u_n-x_n\to0$ and $\nabla g$ is uniformly continuous on bounded subsets of $U$ by $(G1)$--$(G2)$,
\[
\|\nabla g(u_n)-\nabla g(x_n)\|\to0.
\]
Combining this with \eqref{residual1}, we obtain
\begin{equation}\label{residual2}
\|\nabla g(u_n)-\nabla g(J_*\widehat T_nu_n)\|\to0.
\end{equation}

Fix $j\in\mathbb N$. The recurrence of $\tau_T$ provides an infinite subsequence
$\{n_q^{(j)}\}$ for which $\widehat T_{n_q^{(j)}}=T_j$ for every $q$. Hence
\[
\|\nabla g(u_{n_q^{(j)}})-\nabla g(J_*T_j u_{n_q^{(j)}})\|\to0.
\]
Both dual sequences lie in the bounded set $\mathcal{V}^{*}$. Lemma~\ref{dualresidual} therefore yields
\[
\|u_{n_q^{(j)}}-J_*T_j u_{n_q^{(j)}}\|\to0.
\]
Since $u_{n_q^{(j)}}\to x^{*}$, it follows that $J_*T_j u_{n_q^{(j)}}\to x^{*}$. Because $T_j$ is
$J_*$-closed, we have $J_*T_jx^{*}=x^{*}$. Applying $J$ yields $T_jx^{*}=Jx^{*}$; hence
$x^{*}\in F_J(T_j)$. Since $j$ was arbitrary,
\[
x^{*}\in\bigcap_{j=1}^{\infty}F_J(T_j).
\]

\noindent\textbf{Step 5.} We show that $x^{*}\in\bigcap_{l=1}^{\infty}\EP(\Theta_l)$.

Since $z_n=R_{r_n}^{\widehat\Theta_n}x_n$, we have
\begin{equation}\label{EPstep}
\widehat\Theta_n(z_n,y)+\frac1{r_n}\langle y-z_n,\nabla g(z_n)-\nabla g(x_n)\rangle\geq0,
\qquad \forall y\in C.
\end{equation}
Since $z_n\to x^{*}$, $x_n\to x^{*}$, and $r_n\geq r_0>0$, the continuity of $\nabla g$ yields
\[
\frac{\|\nabla g(z_n)-\nabla g(x_n)\|}{r_n}\to0.
\]
Fix $l\in\mathbb N$. The control condition for $\tau_\Theta$ provides an infinite subsequence
$\{n_j^{(l)}\}$ such that $\widehat\Theta_{n_j^{(l)}}=\Theta_l$ for every $j$. By the monotonicity of
$\Theta_l$, \eqref{EPstep} gives
\[
\Theta_l(y,z_{n_j^{(l)}})
\leq \frac1{r_{n_j^{(l)}}}\langle y-z_{n_j^{(l)}},
\nabla g(z_{n_j^{(l)}})-\nabla g(x_{n_j^{(l)}})\rangle,
\qquad \forall y\in C.
\]
The right-hand side tends to zero as $j\to\infty$. Since $z_{n_j^{(l)}}\to x^{*}$ and
$\Theta_l(y,\cdot)$ is lower semicontinuous on $C$, we obtain
\[
\Theta_l(y,x^{*})
\leq\liminf_{j\to\infty}\Theta_l(y,z_{n_j^{(l)}})
\leq\limsup_{j\to\infty}\Theta_l(y,z_{n_j^{(l)}})
\leq0,
\qquad \forall y\in C.
\]
For $y\in C$ and $t\in(0,1]$, set $v_t:=x^{*}+t(y-x^{*})\in C$. The preceding inequality with
$y$ replaced by $v_t$ gives $\Theta_l(v_t,x^{*})\leq0$. Therefore, by $(E1)$ and the convexity part
of $(E4)$,
\[
0=\Theta_l(v_t,v_t)
\leq t\Theta_l(v_t,y)+(1-t)\Theta_l(v_t,x^{*})
\leq t\Theta_l(v_t,y).
\]
Thus $\Theta_l(v_t,y)\geq0$. Applying $(E3)$ with $x=x^{*}$ and $z=y$ gives
\[
0\leq\limsup_{t\downarrow0}\Theta_l(v_t,y)\leq\Theta_l(x^{*},y).
\]
Hence $\Theta_l(x^{*},y)\geq0$ for every $y\in C$, so $x^{*}\in\EP(\Theta_l)$. Since $l$ was
arbitrary,
\[
x^{*}\in\bigcap_{l=1}^{\infty}\EP(\Theta_l).
\]

From Steps 3--5, $x^{*}\in \Sol$.

Finally, since $\Bproj_{\Sol}x\in \Sol\subset C_n$ and $x_n=\Bproj_{C_n}x$, we have
\[
\Dg(x_n,x)\leq \Dg(\Bproj_{\Sol}x,x),
\]
and, since $\Dg(\cdot,x)$ is continuous on $U$, by passing to the limit,
\begin{equation}\label{final1}
\Dg(x^{*},x)\leq \Dg(\Bproj_{\Sol}x,x).
\end{equation}
On the other hand, $x^{*}\in \Sol$, and so the definition of $\Bproj_{\Sol}x$ gives
\begin{equation}\label{final2}
\Dg(\Bproj_{\Sol}x,x)\leq \Dg(x^{*},x).
\end{equation}
From \eqref{final1} and \eqref{final2}, $x^{*}$ is a Bregman projection of $x$ onto $\Sol$. By uniqueness
of the Bregman projection, $x^{*}=\Bproj_{\Sol}x$. This completes the proof.

\noindent\hfill$\square$

\subsection{Proof of the localized convergence theorem}

\noindent\textbf{Proof of Corollary~\ref{localmainthm}.}
Because $g$ is Legendre, the gradient bijection
\[
\nabla g:U\longrightarrow U^{*},\qquad (\nabla g)^{-1}=\nabla g^{*},
\]
together with the three-point identity and the Bregman-projection characterization used in
Theorem~\ref{mainthm}, remains available. The componentwise resolvents and every mapping
$J_*T_j$ send $C$ into $C$. Consequently,
\[
x_n,\ z_n,\ u_n,\ J_*\widehat T_nu_n\in C\subset K_0
\]
for every $n$, while (L1) gives $y_n\in K_1$. Since $E$ is reflexive and $C$ is bounded, every
nonempty closed convex subset of $C$ is weakly compact. The weak lower semicontinuity and strict
convexity of $z\mapsto\Dg(z,x)$ therefore give the existence and uniqueness of each Bregman
projection used in the outer-approximation construction.

We now identify the local substitutes for the global generator assumptions used in the proof of
Theorem~\ref{mainthm}. First, the boundedness conclusions obtained there from the forward and reverse
sublevel clauses of (G5) follow here directly from
\[
x_n,\ z_n,\ u_n,\ J_*\widehat T_nu_n\in K_0,
\qquad y_n\in K_1.
\]
Second, every passage from vanishing Bregman distance to norm convergence uses pairs of points from
$K_0\cup K_1$ and is therefore justified by the sequential total-convexity clause in (L2), which
replaces (G3) on the generated region. Third, all dual variables entering the Bregman mixture belong
to the bounded convex set $K^{*}$ by (L1). Hence the strict dual convexity estimate used to obtain the
fixed-point residual follows from the uniform convexity of $g^{*}$ on $K^{*}$ in (L3), replacing the
corresponding use of (G4). Finally, the conversions of primal and dual residuals use only the uniform
continuity of $\nabla g$ on $K_0$ and of $\nabla g^{*}$ on $K^{*}$, as supplied by (L2)--(L3).
Thus each estimate in the proof of Theorem~\ref{mainthm} is valid on the invariant bounded sets
specified in (L1)--(L3), and the same argument yields $x_n\to\Bproj_{\Sol}x$.

\noindent\hfill$\square$

\section{Conclusion}

The paper establishes strong convergence of a Bregman outer-approximation scheme for a common-point problem with three countable constraint blocks. The convergence analysis does not require an NST-type compatibility condition for the fixed-point family. The Hilbert-space and Alber-functional corollaries provide special cases of the main result in this Bregman framework. 

\section*{Statements and declarations}
\noindent\textbf{Competing interests.} The author declares no competing interests.

\noindent\textbf{Funding.} The author received no specific funding for this work.

\noindent\textbf{Data availability.} No datasets were generated or analyzed during the current theoretical study.

\noindent\textbf{Code availability.} No computer code was generated or analyzed for the present theoretical work.

\end{document}